\documentclass{article}

\usepackage[utf8]{inputenc}
\usepackage[english]{babel}
\usepackage{amsmath,amsthm,amsfonts}
\usepackage{graphicx}
\usepackage{bm}	
\usepackage{tikz}
\usepackage{multirow}
\usepackage{calligra}
\usepackage{mathrsfs}
\usepackage{relsize}
\usepackage{amssymb}

\newtheorem{theorem}{Theorem}[section]
\newtheorem{lemma}[theorem]{Lemma}
\graphicspath{ {./pic-stokes/}}

\renewcommand{\div}{\mathop{\rm div}\nolimits}

\newcommand{\norm}[1]{\left\|#1\right\|}

\title{A conservative local multiscale model reduction technique for Stokes flows in heterogeneous perforated domains}

\author{
Eric T. Chung 
\thanks{
Department of Mathematics,
The Chinese University of Hong Kong (CUHK), Hong Kong SAR.
Email: {\tt tschung@math.cuhk.edu.hk}. 
}
\and
Maria Vasilyeva
\thanks{
Department of Computational Technologies, Institute of Mathematics and Informatics, North-Eastern Federal University, Yakutsk, 677980, Republic of Sakha (Yakutia), Russia \& Institute for Scientific Computation, Texas A\&M University, College Station, TX 77843. 
Email: {\tt vasilyevadotmdotv@gmail.com}.
}
\and
Yating Wang
\thanks{
Department of Mathematics, Texas A\&M University, College Station, TX 77843-3368, USA.
Email: {\tt wytgloria@math.tamu.edu}.
}
}

\begin{document}
\maketitle

\begin{abstract}
In this paper, we present a new multiscale model reduction technique
for the Stokes flows in heterogeneous perforated domains. 
The challenge in the numerical simulations of this problem lies in the fact that
the solution contains many multiscale features and requires
a very fine mesh to resolve all details. 
In order to efficiently compute the solutions, some model reductions are necessary. 
To obtain a reduced model, we apply the generalized multiscale finite element approach,
which is a framework allowing systematic construction of reduced models.
Based on this general framework, 
we will first construct a local snapshot space, which contains many possible multiscale features of the solution.
Using the snapshot space and a local spectral problem, we identify dominant modes in the snapshot space
and use them as the multiscale basis functions. 
Our basis functions are constructed locally with non-overlapping supports, which enhances the sparsity of the resulting linear system.
In order to enforce the mass conservation, we propose a hybridized technique,
and uses a Lagrange multiplier to achieve mass conservation. 
We will mathematically analyze the stability and the convergence of the proposed method.
In addition, we will present some numerical examples
to show the performance of the scheme.
We show that, with a few basis functions per coarse region, one can obtain a solution
with excellent accuracy.
\end{abstract}

\section{Introduction}

Many application problems, such as fluid flow in heterogeneous porous media,  
involve perforated domains (see Figure \ref{fig:domain} for an example of perforated domain) where the perforations can have various sizes and geometries. Due to these features, the solutions of differential equations posed in perforated domains
have multiscale properties. Numerical simulations for these problems are prohibitively expensive, because the computational cost to recover the fine scale properties between perforations
is extremely high. Similar to other types of multiscale problems, some model reduction methods are necessary in order to improve the computational efficiency. 
There are in literature many model reduction techniques that 
are performed on a coarse grid which has much larger length scale compared with the size of perforations, such as numerical homogenization 
(\cite{allaire1991homogenization,maz2010asymptotic,Jikov91,oleinik1996homogenization,fratrovic2015nonlinear,pankratova2015spectral,allaire2015upscaling,bare2015non,gilbert2014acoustic,muntean2012analysis,gilbert2006prototype,gilbert2003acoustic,sanchez1980non}) and multiscale methods (\cite{henning2009heterogeneous, hw97, ReducedCon, egh12, CELV2015,eh09,Iliev_MMS_11, Chu_Hou_MathComp_10, ab05, Babuska,Bris14, brown2014multiscale}). In these approaches, macroscopic equations are formulated on a coarse grid with mesh size independent of the size of perforations. 
While these approaches are excellent in some cases, they are lack of systematic enrichment strategies
in order to tackle problems with more complicated structures. 

The recently developed
Generalized multiscale finite element method (GMsFEM) \cite{egh12,ceh2016adaptive} is a framework that allows systematic enrichment of the coarse spaces and take into account fine scale information for the construction of these spaces. The framework therefore provides a convincing approach to solve problems posed in heterogeneous perforated domains, whose solutions have multiscale features and require sophisticated enrichment techniques. The main idea of GMsFEM is to employ local snapshots to approximate the fine scale solution space, and then identify local multiscale spaces by performing some carefully selected local spectral problems defined in the snapshot spaces. 
The spectral problems give a systematic strategy to identify the dominant modes in the snapshot spaces, and
the dominant modes are selected to form the local multiscale spaces. 
By appropriately choosing the snapshot space and the spectral problem, the GMsFEM requires only a few basis functions per coarse region
in order to obtain solutions with excellent accuracy.
In \cite{CELV2015,chung2016online}, we have developed and analyzed a GMsFEM for elliptic problem, elastic problem and the Stokes problem in perforated domains using
the continuous Galerkin (CG) framework. For this CG approach, we partition the computational domain as a union of overlapping coarse neighborhoods, and 
construct a set of local multiscale basis functions for each coarse neighborhood. 
We also developed an adaptivity procedure based on local residuals to enrich the coarse space by adaptively adding new basis functions. 
However, one drawback of the CG approach is the need to multiply each basis function by 
a partition of unity function. This step may modify the local heterogeneity and cause some difficulties. 

In this paper, we propose a new GMsFEM for problems in perforated domains using a 
discontinuous Galerkin (DG) approach. The use of the DG approach in GMsFEM has been successfully developed for many problems, such as the elliptic equations and the wave equations with heterogeneous coefficients (\cite{chung2015online, AdaptiveGMsDGM, ElasticGMsFEM, WaveGMsFEM, chung2015fracture}). 
The main feature of the DG approach is that the basis functions are constructed locally for each non-overlapping coarse region.
This fact allows much more flexibility in the design of the coarse mesh and in the choice of the local multiscale space. 
Another advantage of the DG approach is that there is no need to construct and use any partition of unity functions. 
We will, in this paper, consider a GMsFEM based on a DG approach for 
the Stokes flows in heterogeneous perforated domains. 
To construct the multiscale basis functions, we will obtain the local snapshots by solving the Stokes equations
for each non-overlapping coarse region with some suitable boundary conditions. 
Then, we will construct local spectral problems and identify dominant modes in the snapshot space. 
The multiscale space is obtained by the span of all these dominant modes.
Furthermore, it is important to note that the mass conservation is a crucial property for the Stokes flow. 
By the construction of the basis functions, the multiscale solution satisfies some local mass conservation property
within coarse regions. 
However, mass conservation does not in general hold globally in the coarse grid level.
To tackle this issue, we construct a hybridized scheme and introduce additional pressure variables
on the coarse grid edges. This additional pressure variable serves as a Lagrange multiplier
to enforce the mass conservation property in the coarse grid level. 
Thus, our new GMsFEM provides solutions using only few basis functions per coarse regions, and
having both local and global mass conservation. 


To investigate the performance of our proposed method, 
we will numerically study the Stokes problem in various perforated domains (see Figure \ref{fig:mesh}) with various choices of boundary conditions and forcing terms. We will present the construction of the snapshot space using both the standard 
and the oversampling approaches (\cite{eglp13oversampling, randomized2014}). Local spectral decompositions are also proposed for various approaches of snapshots correspondingly. Moreover, when constructing multiscale basis, we will test the use of different shapes of coarse blocks for different types of perforated domains. Numerical results are presented and convergence of the method is analyzed. Moreover, we will numerically show that the local mass conservation property is satisfied by the multiscale solution. Our numerical results show that we can approximate the solution using a fairly small degrees of freedom. In addition, the oversampling technique can be particularly helpful and improve the accuracy and the convergence.

We organize the paper as follows. In Section \ref{sec:prelim}, we state the model problem and define the fine and coarse scale discretizations. We present the detailed constructions of the snapshot space and the offline space in Section \ref{sec:constr}. Section \ref{sec:numerical} presents the numerical results for various examples. We analyze the stability and the convergence of our method in Section \ref{sec:analysis}. A conclusion is given at the end of the paper.

\section{Problem settings}\label{sec:prelim}
In this section, we state the Stokes flow in heterogeneous perforated domains and introduce 
some notations.
Let $\Omega \subset \mathbb{R}^n$ ($n = 2,3$) be a bounded domain. We define a perforated domain $\Omega^{\epsilon} \subset \Omega$ with 
a set of perforations denoted by $\mathcal{B}^{\epsilon}$, that is, $\Omega^{\epsilon} = \Omega \backslash \mathcal{B}^{\epsilon}$.
We assume that the set $\mathcal{B}^{\epsilon}$ contains circular perforations with various sizes and positions. 
An illustration of a perforated domain is shown in Figure \ref{fig:domain}. 
We notice that the variable sizes and positions of these perforations lead to some multiscale features in the solutions of the problems posed in perforated domains. Given 
the source function $f$ and two boundary functions $g_D, g_N$, we consider the following Stokes flow in the perforated domain $\Omega^{\epsilon}$:
\begin{equation}\label{eq:stokes}
\begin{aligned}
-\Delta u + \nabla p &= f, \quad &\text{in } \Omega^{\epsilon} \\
\div u &=0, \quad &\text{in } \Omega^{\epsilon}
\end{aligned}
\end{equation}
subject to boundary condition $u = g_D$ on $\Gamma_D$, and $(\nabla u - pI)n = g_N$ on $\Gamma_N$, where $\Gamma_D \cup \Gamma_N = \partial \Omega^{\epsilon}$, $n$ is the unit outward normal vector on $\partial \Omega^{\epsilon}$ and $I$ is the $n\times n$ identity matrix. 
The unknown variable $u$ denotes the fluid velocity and $p$ denotes the fluid pressure. Since $p$ is uniquely defined up to a constant, we assume that $\int_{\Omega^{\epsilon}} p = 0$,
so that the problem (\ref{eq:stokes}) has a unique solution. 

\begin{figure}[!h]
	\begin{center}
	    \includegraphics[width=0.3\textwidth]{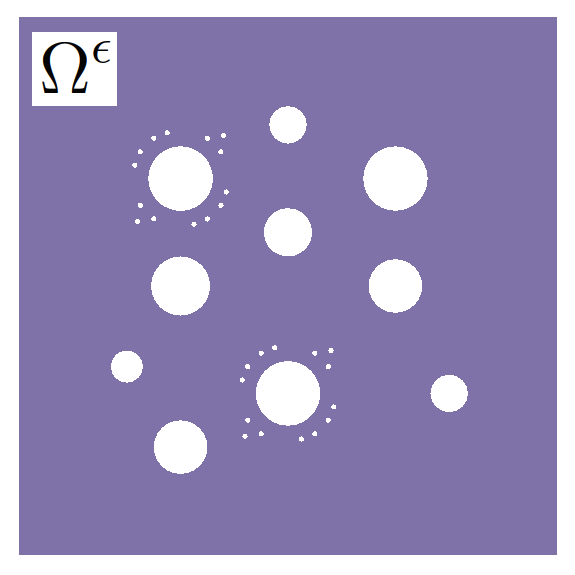}
    \end{center}    
    \caption{An illustration of a perforated domain.}
    \label{fig:domain}
\end{figure}

Let $V(\Omega^{\epsilon}) = H_0^1(\Omega^{\epsilon})^n$ and $Q(\Omega^{\epsilon}) = L^2_0(\Omega^{\epsilon})$,
where $ L^2_0(\Omega^{\epsilon})$ is the set of $L^2$ functions defined in $\Omega^{\epsilon}$ with zero mean.
The variational formulation of \eqref{eq:stokes} is given by: find $u \in V(\Omega^{\epsilon})$ and $p \in Q(\Omega^{\epsilon})$ such that 
\begin{equation}\label{eq:stokes-weak}
\begin{aligned}
 a(u,v) + b(v, p) &= (f, v), \quad \text{for all} \quad v \in V(\Omega^{\epsilon})\\
 b(u,q) &= 0, \quad \text{for all} \quad q \in Q(\Omega^{\epsilon})
\end{aligned}
\end{equation}
where \[
a(u, v) = \int_{\Omega^{\epsilon}} \nabla u : \nabla v, \quad b(v, q) = - \int_{\Omega^{\epsilon}} q \, \div v 
\]
and 
\[
(f,v) = \int_ {\Omega^{\epsilon}} f v.
\]
It is well known that there is a unique weak solution to \eqref{eq:stokes-weak} (see for example \cite{Brezzi_Fortin_book}).

For the numerical approximation of the above problem, we first introduce the notations of fine and coarse grids.
Let $\mathcal{T}^H$ be a coarse-grid partition of the domain
$\Omega^{\epsilon}$ with mesh size $H$.  
We assume that this coarse mesh does not necessarily resolve the full details of the perforations. 
By using a conforming
refinement of the coarse mesh $\mathcal{T}^H$, we can obtain
a fine mesh $\mathcal{T}^h$ of $\Omega^{\epsilon}$ with mesh size $h$.
Typically, we assume that $0 < h \ll H < 1$, and that the fine-scale mesh $\mathcal{T}^h$
is sufficiently fine to fully resolve the small-scale information of the domain, and $\mathcal{T}^H$ is a coarse mesh containing many fine-scale feature. 
We use the notations $K$ and $E$ to denote a coarse element and a coarse edge in the coarse grid $\mathcal{T}^H$.

We let $\mathcal{E}^H$ be the set of edges in $\mathcal{T}^H$.
We write $\mathcal{E}^H = \mathcal{E}^H_{int} \cup \mathcal{E}^H_{out}$,
where $\mathcal{E}^H_{int}$ is the set of interior edges
and $\mathcal{E}^H_{out}$ is the set of boundary edges. 
For each interior edge $E\in \mathcal{E}^H_{int}$, we define the jump $[u]$ and the average $\{u\}$ of a function $u$ by
\[
[u]_E = u|_{K^{+}} - u|_{K^{-}}, \;\; \{u\}_E = \frac{u|_{K^{+}} + u|_{K^{-}}}{2},
\]
where $K^+$ and $K^-$ are the two coarse elements sharing the edge $E$,
and the unit normal vector $n$ on $E$ is defined so that $n$ points from $K^+$ to $K^-$.
For $E \in \mathcal{E}^H_{out}$, we define
\[
[u]_E = u|_E,\;\;  \{u\}_E = u|_E.
\]

Next we introduce our DG scheme. Similar to the standard derivation of DG formulations \cite{ABCM_unified_2002, Ewing_WangYang2003, Laz_PShoebV_2003, Laz_Tom_Vas_2001}, the main idea is to consider the problem in each element $K$ in the coarse mesh, and impose boundary conditions weakly on $\partial K$ using the value of the velocity function in the neighboring elements. In addition, a penalizing term which penalize the jump of velocity will be introduced.  After obtaining the local problems in each element, one can sum over all elements to get the global DG scheme. Remark that, in our approach, we will only assume discontinuity across the coarse edges, but use the standard continuous element inside coarse blocks. In this work, we also add an additional Lagrange multiplier in order to impose local mass conservation on the coarse elements. The details are given as follows.

We start with the definitions of the approximation spaces. 
We let $Q_H$ be the piecewise constant function space for the approximation of the pressure $p$. 
That is, the restriction of the functions of $Q_H$ in each coarse element is a constant. 
In addition, we will define a piecewise constant space $\widehat{Q}_H$ for the approximation of the pressure $ \widehat{p}$, which is defined on the set of coarse edges $\mathcal{E}^H$.
That is, the functions in $\widehat{Q}_H$ are defined only in $\mathcal{E}^H$
and the restriction of the functions of $\widehat{Q}_H$ in each coarse edge is a constant.
We remark that this additional pressure space is used to enforce local mass conservation in the coarse grid level. 
Moreover, we define $V_H$ as the multiscale velocity space, which contains a set of basis functions supported in each coarse block $K$. To obtain these basis functions, we will solve some local problems in each coarse block with various Dirichlet boundary conditions to form a snapshot space
and use a spectral problem to perform a dimension reduction. 
The details for the construction of this space will be presented in the next section.

For our GMsFEM using a DG approach, 
we define the bilinear forms
\begin{equation}
a_{\text{DG}}(u,v) = \int_{\Omega^{\epsilon}} \nabla u : \nabla v
- \sum_{E\in\mathcal{E}^H} \Big( \int_E \{ (\nabla u) \, n \}\cdot [v] + \{ (\nabla v)  n\}  \cdot [u] \Big)
+ \frac{\gamma}{h} \sum_{E \in\mathcal{E}^H} \int_E [u]\cdot [v],
\end{equation}
\begin{equation}
b_{\text{DG}}(v,q, \widehat{q}) = - \sum_{K\in\mathcal{T}^H} \int_K q \, \div v  +  \sum_{E \in\mathcal{E}^H} \int_E \widehat{q} \, ([v]\cdot n).
\end{equation}
Then,
we will find the multiscale solution $(u_H, p_H, \widehat{p}_H)\in V_H \times Q_H \times \widehat{Q}_H$ such that 
\begin{equation}\label{eq:coarse-scale}
\begin{aligned}
a_{\text{DG}}(u_H,v) + b_{\text{DG}}(v,p_H,\widehat{p}_H) &= (f,v) + \int_{\Gamma_D} \Big(\frac{\gamma}{h} g_D \cdot v -  ((\nabla v) \, n) \cdot g_D \Big) 
+\int_{\Gamma_N} g_N \cdot v,\\
b_{\text{DG}}(u_H,q, \widehat{q}) &= \int_{\Gamma_D} (g_D\cdot n) \, \widehat{q},
\end{aligned}
\end{equation}
for all $v \in V_H,  \, q\in Q_H,  \, \widehat{q}\in\widehat{Q}_H$.
The derivation of the above scheme follows the standard DG derivation procedures \cite{ABCM_unified_2002, Ewing_WangYang2003, Laz_PShoebV_2003, Laz_Tom_Vas_2001}.
We notice that the role of the variable $\widehat{p}_H$ is to enforce mass conservation on coarse elements. 
In particular, taking $q=0$ in (\ref{eq:coarse-scale}), we have
\begin{equation*}
\int_E \widehat{q}\, [u_H]\cdot n = 0, \quad \forall E\in\mathcal{E}^H_{int}, \quad \forall \, \widehat{q} \in \widehat{Q}_H.
\end{equation*}
This relation implies that
\begin{equation*}
\int_K q \, \div u_H = 0, \quad \forall K \in \mathcal{T}^H, \quad \forall q\in Q_H. 
\end{equation*}
The above is the key to the mass conservation, and we will discuss more in the numerical results section.

We will show the accuracy of our method by comparing the multiscale solution to a reference solution,
which is computed on the fine mesh. 
To find the reference solution $(u_h,p_h,\widehat{p}_h)$, we will solve the following system
\begin{equation}\label{eq:fine-scale}
\begin{aligned}
a_{\text{DG}}(u_h,v) + b_{\text{DG}}(v,p_h, \widehat{p}_h) &= (f,v)+  \int_{\Gamma_D} \Big(\frac{\gamma}{h} g_D \cdot v -  ((\nabla v) \, n) \cdot g_D \Big) 
+\int_{\Gamma_N} g_N \cdot v, \\
b_{\text{DG}}(u_h,q, \widehat{q}) &=\int_{\Gamma_D} (g_D\cdot n) \, \widehat{q},\\
\end{aligned}
\end{equation}
for all $v\in V_h^{\text{DG}}, q\in Q_H, \widehat{q}\in \widehat{Q}_H $.
We note that the reference velocity $u_h$ belongs to
the fine scale velocity space $V_h^{\text{DG}} = \{v \in L^2(\Omega^{\epsilon})| \;\; v|_{K} \in C^0(K)^2 \; \text{for every } K\in\mathcal{T}^H, \;
v|_{K} \in (\mathbb{P}_1(T))^2 \; \text{for every } K \in \mathcal{T}^h\}$. The space $V_h^{\text{DG}} $ contains functions which are piecewise linear in each fine-grid element $K$ and are continuous along the fine-grid edges, but are discontinuous across coarse grid edges.
Moreover, the reference pressure $p_h$ and $\widehat{p}_h$ belongs to the coarse scale pressure space $Q_H$ and $\widehat{Q}_H$ respectively.
Notice that the pressure $p_h$ is determined up to a constant, we will achieve the uniqueness by requiring the averaging value of pressure over whole domain is zero. 
We remark that this reference solution $(u_h,p_h,\widehat{p}_h)$ is obtained using the coarse scale pressure spaces $Q_H$ and $\widehat{Q}_H$
since we only consider multiscale solutions and reduced spaces for the velocity.
The true fine scale solution $(u_{\text{fine}},p_{\text{fine}},\widehat{p}_{\text{fine}})$ can be defined by
\begin{equation*}
\begin{aligned}
a_{\text{DG}}(u_{\text{fine}},v) + b_{\text{DG}}(v,p_{\text{fine}}, \widehat{p}_{\text{fine}}) &= (f,v)+  \int_{\Gamma_D} \Big(\frac{\gamma}{h} g_D \cdot v -  ((\nabla v) \, n) \cdot g_D \Big) 
+\int_{\Gamma_N} g_N \cdot v, \\
b_{\text{DG}}(u_{\text{fine}},q, \widehat{q}) &=\int_{\Gamma_D} (g_D\cdot n) \, \widehat{q},\\
\end{aligned}
\end{equation*}
for all $v\in V_h^{\text{DG}}, q\in Q_h, \widehat{q}\in \widehat{Q}_h $,
where $Q_h$ and $\widehat{Q}_h$ are suitable fine scale spaces. 
One can see that $(u_{\text{fine}}, p_{\text{fine}})$ will converge to the exact solution $(u, p)$ in the energy norm as the fine mesh size $h \rightarrow 0$. 
Moreover, one can show that 
\begin{equation*}
\|u_{\text{fine}} - u_h\|_A^2 \leq C \inf_{q\in Q_H, \widehat{q} \in \widehat{Q}_H} \| (p_{\text{fine}}-q, \widehat{p}_{\text{fine}} - \widehat{q})\|_Q^2
\end{equation*}
where the norms are defined in (\ref{eq:A-norm}) and (\ref{eq:q-norm}).
Thus, the reference solution defined in (\ref{eq:fine-scale})
can be considered as the exact solution up to a coarse scale approximation error.

\section{Construction of multiscale velocity space}\label{sec:constr}

In this section, we will present the construction of the multiscale space $V_H$ for the coarse scale approximation of velocity. To construct the coarse scale velocity space, we will follow the general idea of GMsFEM \cite{egh12, eglp13oversampling}, which contains two stages: (1) the construction of snapshot space, and (2) the construction of offline space. In the first stage, we will obtain the snapshot space, which contains a rich set of functions containing possible features in the solution. These snapshot functions are solutions of some local problems subject to all possible boundary conditions up to the fine grid resolution. Notice that for the generalized multiscale DG scheme proposed in \cite{AdaptiveGMsDGM, MsDG}, one solves the local problems in each coarse block. Thus the resulting system is much smaller compared with that of the CG approach \cite{CELV2015, ElasticGMsFEM, AdaptiveGMsFEM}, where the local problems are solved in each overlapping coarse neighborhood. Next, in order to reduce the dimension of the solution space, we will use a space reduction technique to choose the dominated modes in the snapshot space. This procedure is achieved by defining proper local spectral problems. The resulting reduced order space is called the offline space and will be used for coarse scale velocity approximation. Note that for approximating pressure on the coarse grid, we will use piecewise constant functions as defined before. 
In Section \ref{snap}, we will present the construction of the snapshot space,
and in Section \ref{off}, we will present the construction of the offline space.

\subsection{Snapshot space}\label{snap}
We will construct local snapshot basis in each coarse block  $K_i, (i = 1, \cdots, N)$, where $N$ is the number of coarse blocks in $\Omega^{\epsilon}$. The local snapshot space consists of functions which are solutions $u \in V_h(K_i)$ of 
\begin{equation}\label{eq:hmex}
\begin{aligned}
-\Delta u + \nabla p &= 0, \quad &\text{in } K_i \\
\div u &=c, \quad &\text{in }  K_i
\end{aligned}
\end{equation}
with $u=\delta_i^k$ on $\partial K_i$, ($k= 1, \cdots, M_i$), where $M_i$ is the number of fine grid nodes on the boundary of $K_i$,
and $\delta_i^k$ is the discrete delta function defined on $\partial K_i$.
The above problem (\ref{eq:hmex}) is solved on the fine mesh using some appropriate approximation spaces.
For instances, we take the space $V_h(K_i)$ to be the standard conforming piecewise linear finite element space with respect to the fine grid on $K_i$. 
Note that the constant $c$ in \eqref{eq:hmex} is chosen by the compatibility condition, that is, $c = \frac{1}{|K_i|}\int_{\partial K_i} \delta_i^k \cdot n ds$.

Take these $M_i$ velocity solutions of \eqref{eq:hmex} and denote them by $\psi_k^{i, \text{snap}}(k = 1, \cdots, M_i)$, we get the local snapshot space
\[
V^i_{\text{snap}} = \text{span}\{\psi_1^{i, \text{snap}}, \cdots, \psi_{M_i}^{i, \text{snap}} \}.
\]
Combining all the local snapshots, we can form the global snapshot space, that is
\[
V_{\text{snap}} = \text{span}\{\psi_k^{i, \text{snap}}, \quad, 1\leq k\leq M_i, 1\leq i \leq N \}.
\]

In the above construction, the local problems are solved for every fine grid node on $\partial K_i$. One can also apply the oversampling strategy \cite{Babuska, eglp13oversampling} in order to reduce the boundary effects. Applying this strategy, one can solve the local problem for each fine node on the boundary of the oversampled domain. An illustration of the original local domain $K$ and the oversampled local domain $K^{+}$ are shown in Figure \ref{fig:illus-os}. Notice that in Figure \ref{fig:illus-os}, we present the triangular coarse grid in perforated domain with small inclusions on the left, and rectangular coarse grid perforated domain with multiple sizes of inclusions on the right. We will solve the local problem in an enlarged domain $K_i^{+}$ of $K_i$, 
\begin{equation*}
\begin{aligned}
-\Delta u + \nabla p &= 0, \quad &\text{in } K_i^{+} \\
\div u &=c, \quad &\text{in }  K_i^{+}
\end{aligned}
\end{equation*}
with $u=\delta_i^k$ on $\partial K_i^{+}$, where $k= 1, \cdots, M_i^{+}$, where $M_i^{+}$ is the number of fine nodes on the boundary of $K_i^{+}$.  After removing linear dependence among these basis by POD, we denote the linearly independent functions  by $\psi_k^{+, i}, \;(i = 1, \cdots, \tilde{M}_i)$. Note that the velocity solutions of these local problems are supported in the larger domain $K_i^{+}$. There are usually several following choices for identification of basis.  One of the straight forward way is that, we can restrict the basis  $\psi_k^{+, i}$ on $K_i$ to form the snapshot basis, i.e. $\psi_k^{i, \text{snap}} = \psi_k^{+, i}|_{K_i}$. Then the span of these basis function $\psi_k^{i, \text{snap}}$ will form our new snapshot space. In this case, the local reduction will be performed in  $K_i$. Another choice is that, one can keep the snapshot basis $\psi_k^{+, i}$ without restricting on $K_i$. But in this case, one needs to solve the offline basis also in the oversampled domain $K_i^{+}$ and finally restrict the offline basis on the original local domain $K_i$. It is known that these oversampling methods can improve the accuracy of our multiscale methods (\cite{eglp13oversampling}). 

We remark that one can also use the idea of randomized snapshots (as in \cite{randomized2014}) and reduce the computational cost substantially. In randomized snapshots approach, instead of solving the local problem for each fine node on the boundary of oversampled local domain, one only computes a few snapshots in each oversampled domain $K_i^{+}$ with several random boundary conditions. These random boundary functions are constructed by independent identically distributed (i.i.d.) standard Gaussian random vectors defined on the fine degrees of freedoms on the boundary. The randomized snapshot requires much fewer calculations to achieve a good accuracy compared with the standard snapshot space. 


\begin{figure}[!h]
	\begin{center}
	    \includegraphics[width=0.7\textwidth]{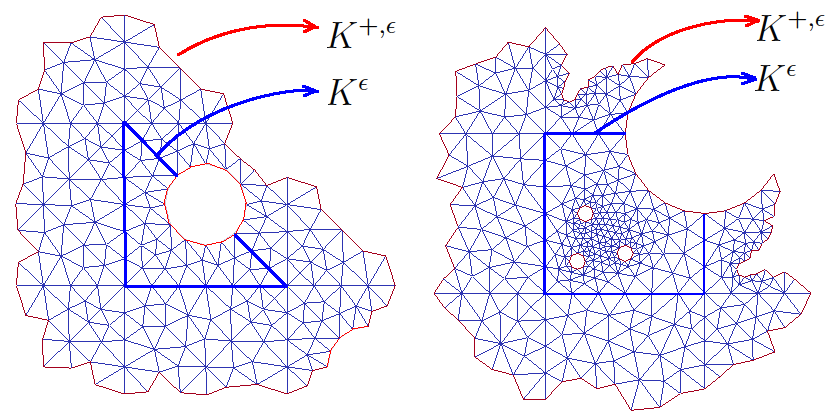}
    \end{center}   
    \caption{Illustration of oversampling domain. Left: Oversampling of a triangular coarse block for perforated domain with small inclusions. Right: Oversampling of a rectangular coarse block for perforated domain with multiple sizes of inclusions.}
    \label{fig:illus-os}
\end{figure}

\subsection{Offline space}\label{off}
In this section, we will perform local model reduction on the snapshot space by solving some local spectral problems. The reduced space consists of the important modes in the snapshot space, and is called the offline space. The coarse scale approximation of velocity solution will be obtained in this space. We have multiple choices of local spectral problems given the various constructions of snapshot space presented in the previous section. 

First of all, if the snapshot basis obtained in the previous section is supported in each coarse element $K_i$, we will solve for $(\lambda, \Phi)$ from the generalized eigenvalue problem in the snapshot space 
\begin{equation}\label{eq:eig}
A \Phi = \lambda S \Phi
\end{equation}
where $A$ is the matrix representation of the bilinear form $a_i(u, v)$ and $S$ is the matrix representation of the bilinear form $s_i(u, v)$.
The choices for $a_i$ and $s_i$ are based on the analysis. 
In particular, we take
\begin{align*}
a_i(u,v)&= \int_{K_i} \nabla u : \nabla v, \\
s_i(u,v)&= \frac{\lambda}{H} \int_{\partial K_i} u\cdot v,
\end{align*}
where we remark that the integral in $s_i(u,v)$ is defined on the boundary of the coarse block. In this case, the number of the spectral problem  equals the number of coarse blocks. 

We arrange the eigenvalues of (\ref{eq:eig}) in increasing order. 
We will choose the first few eigenvectors corresponding to the first few small eigenvalues. Using these eigenvectors as the coefficients, we can form our offline basis. More precisely, assume we arrange the eigenvalues in increasing order
\[
\lambda_1^{(i)} < \lambda_2^{(i)} < \cdots < \lambda_{M_i}^{(i)}. 
\]
The corresponding eigenvectors are denoted by $\Phi_k^{(i)} = (\Phi_{kj}^{(i)})_{j=1}^{M_i}$, where $\Phi_{kj}^{(i)}$ is the $j$-th component of the eigenvector. 
We will take the first $L_i \leq M_i$ eigenvectors to form the offline space, that is, the offline basis functions can be constructed as
\[
\phi_k^{i, \text{off}} = \sum_{j=1}^{M_i}\Phi_{kj}^{(i)} \psi_k^{i, \text{snap}}, \;\; k = 1, \cdots, L_i. 
\]

On the other hand, 
one can use the snapshot basis $\psi_k^{+, i}$ (using oversampling strategy) without restricting on $K_i$ in the space reduction process. To be more specific, since the snapshot basis are supported in the oversampled domain  $K_i^{+}$, we will need another set of spectral problems, namely
\begin{equation}\label{eq:eig-os}
A^+ \Phi^+ = \lambda S^+ \Phi^+
\end{equation}
where $A^+$ and $S^+$ are the matrix representations of the bilinear forms $a_{+,i}(u, v)$ and $s_{+,i}(u, v)$ respectively. 
Similar as before, we can choose $a_{+,i}, s_{+,i}$ as follows
\begin{align*}
a_{+,i}(u,v)&= \int_{K_i^{+}} \nabla u : \nabla v, \\
s_{+,i}(u,v)&= \frac{\lambda}{H} \int_{\partial K_i^{+}} u\cdot v.
\end{align*}
We then arrange the eigenvalues in increasing order
\[
\lambda_1^{(i)} < \lambda_2^{(i)} < \cdots < \lambda_{M_i^+}^{(i)}. 
\]
The corresponding eigenvectors are denoted by $\Phi_k^{+,(i)}$. 
We will take the first $L_i \leq M_i^+$ eigenvectors to form a basis supported in $K_i^{+}$
\[
\phi_k^{+, i} = \sum_{j=1}^{M_i^+}\Phi_{kj}^{+,(i)} \psi_k^{+, i}, \;\; k = 1, \cdots, L_i. 
\]
Then we will obtain our offline basis by restricting $\phi_k^{+, i}$ on $K_i$, namely
\[
\phi_k^{i, \text{off}} = \phi_k^{+, i}|_{K_i}.
\]

Now we can finally form the local offline space, which is the span of these basis functions
\[
V^{i}_{\text{off}} = \text{span}\{\phi_1^{i, \text{off}}, \cdots, \phi_{L_i}^{i, \text{off}}\}.
\]
The global offline space $V_{\text{off}}$ is the combination of the local ones, i.e.
\[
V_{\text{off}} = \text{span}\{\phi_k^{i, \text{off}},  \quad, 1\leq k\leq L_i, 1\leq i \leq N\}.
\]
This space will be used as the coarse scale approximation space for velocity $V_H := V_{\text{off}}$. 


\section{Numerical results}\label{sec:numerical}
In this section we will present numerical results of our method for various types of perforations, boundary conditions and sources. We will illustrate the performance of our method using two kinds of perforated domains: (1) perforated domain with small inclusions and (2) perforated domain with big inclusions as well as some extremely small inclusions, see Figure \ref{fig:mesh}. 
We will also illustrate the performance of the oversampling strategy. 

We set $\Omega = [0,1] \times [0,1]$. The computational domain is discretized coarsely using uniform triangulation for domain with small inclusions (Figure \ref{fig:mesh}, left), and uniform rectangle coarse partition for domain with big inclusions (Figure \ref{fig:mesh}, right). The coarse mesh size $H=\frac{1}{10}$. For the fine scale discretization, the size of the system is $69146$ for domain with small inclusions (Figure \ref{fig:mesh}, left) and $91588$ for domain with multiple size of inclusions (Figure \ref{fig:mesh}, right).

\begin{figure}[!h]
	\begin{center}
	    \includegraphics[width=0.35\textwidth]{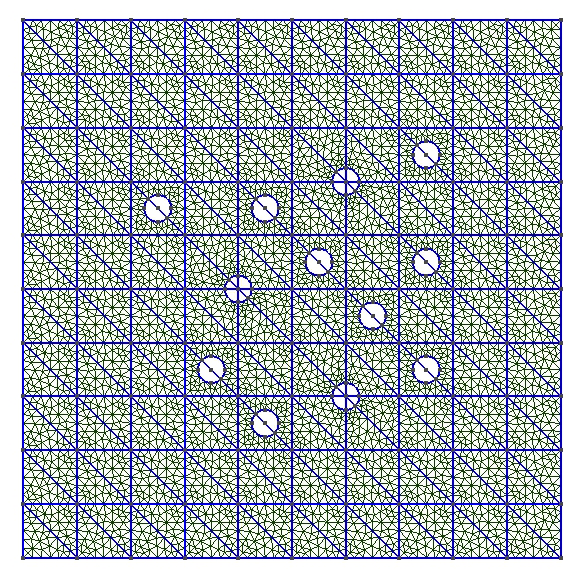}
        \includegraphics[width=0.35\textwidth]{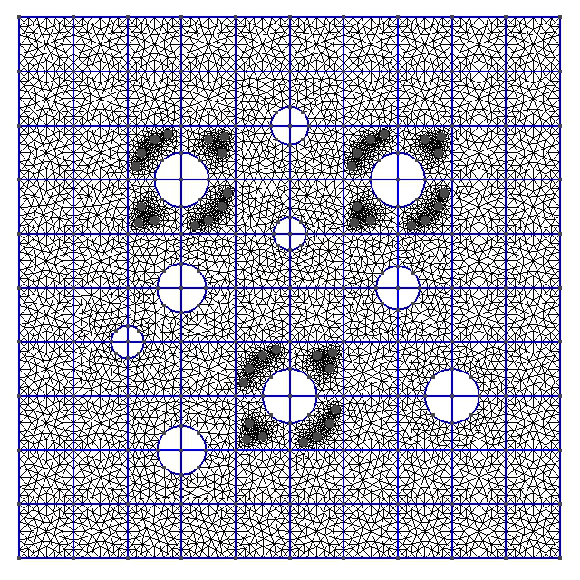}
    \end{center}    
    \caption{Illustration of the perforated domain with fine and coarse mesh. Left: perforated domain with small inclusions.  Right: perforated domain with multiple sizes of inclusions.}
    \label{fig:mesh}
\end{figure}

We will consider two different boundary conditions and force terms:
\begin{itemize}
\item \textit{Example 1}: Source term $f = (0,0)$, boundary condition $u = (1, 0)$ on $\partial{\Omega}$ and $u = (0, 0)$ on $\partial{\mathcal{B}^{\epsilon}}$. 
\item \textit{Example 2}: Source term $f = (1,1)$, boundary condition $\frac{\partial u}{\partial n}-pn = (0, 0)$ on $\partial{\Omega}$ and $u = (0, 0)$ on $\partial{\mathcal{B}^{\epsilon}}$. 
\end{itemize}

The errors will be measured in relative $L^2$, $H^1$ and $DG$ norms for velocity, and $L^2$ norm for pressure
\begin{align*}
||e_u||_{L^2} & =\frac{ \norm{u_h-u_H}_{L^2(\Omega^{\epsilon})}}{\norm{u_h}_{L^2(\Omega^{\epsilon})}} ,\\ 
||e_u||_{H^1} &=\frac{ \norm{u_h-u_H}_{H^1(\Omega^{\epsilon})} }{\norm{u_h}_{H^1(\Omega^{\epsilon})}} , \quad
||e_u||_{DG} = \frac{ \sqrt{a_{\text{DG}}(u_h - u_H, u_h-u_H)}} { \sqrt{a_{\text{DG}}(u_h, u_h)} }, \\
||e_p||_{L^2} &= \frac{ \norm{\bar{p}_h-p_H}_{L^2(\Omega^{\epsilon})} } {\norm{\bar{p}_h}_{L^2(\Omega^{\epsilon})}}.
\end{align*}
where $\bar{p}_h$ is the cell average of the fine scale pressure, that is,
$\bar{p}_h = \frac{1}{|K_i|} \int_{K_i} p_h $ for all $K_i \in {\cal T}^{H}$.


\subsection{Perforated domain with small inclusions}
In this section, we show the numerical results for the Stokes problem in perforated domain with small inclusions (left of Figure~\ref{fig:mesh}), see Table \ref{err-sh-1} for Example 1 and Table \ref{err-sh-2} for Example 2. Remark that the fine scale system has size $69146$, while our coarse scale systems only have size $1280-6880$ when we take 4 to 32 basis, which are much smaller. 
We will first take a look at the numerical behavior for the first example, where we take Dirichlet boundary conditions $u = (1,0)$ on the global boundary, and $u = (0,0)$ on the boundary of inclusions. The force term $f = (0,0)$. In Table \ref{err-sh-1}, we observe that the errors reduce substantially when we add more than 4 basis in each coarse block. For example, when we construct basis without oversampling,  the $L^2$ velocity error reduce from $33.2\%$ to $6.5\%$ when the number of basis increase from 4 to 8. Moreover, the energy error for velocity is $28\%$ and the $L^2$ error for pressure is $12\%$ as we take 32 offline basis for non-oversampling case. To get a faster convergence, we employ oversampling strategy when calculating the basis, that is, we solve the local problems in an oversampled coarse domain and then restrict the local velocity solution to the original coarse block to form our snapshot basis. In our numerical example, the oversampled domain is the original coarse block plus four fine cells layers neighboring the original domain. We can see that, the oversampling case gives us better accuracy with respect to velocity energy error and pressure error. For instance, the velocity energy error reduces from $25\%$ to $18\%$ and the pressure error decreased from $12\%$ to $2\%$ when the number of offline basis is 32 comparing the non-oversampling with oversampling case.

For the second example in perforated domain with small inclusions, we take Neumann boundary condition $\frac{\partial u}{\partial n}-pn = (0, 0)$ on the global boundary and Dirichlet condition $u = (0, 0)$ on the boundary of inclusions. The convergence history is shown in Table \ref{err-sh-2}. From this table, we find that the velocity $L^2$ error reduce from $39.9\%$ to $7.6\%$, and the pressure error reduce from  $69.8\%$ to $5.0\%$ when the basis number increase from 4 to 8 for the non-oversampling case. Moreover, the velocity $L^2$ error reduces to $4.9\%$ when we take 32 basis. We also observe that the oversampling strategy works efficiently to speed up the convergence rate for both the $L^2$ error and the energy error for velocity. For example, the velocity $L^2$ error is $7.6\%$ when we take 8 basis in non-oversampling case, however, it is only $2.6\%$ when we take the same number of basis in oversampling case. The velocity $H^1$ error reduce from $30.5\%$(for non-oversampling case) to $17.4\%$(for oversampling case) when we take 32 basis.  
In addition, we check the local mass conservation and present the numerically computed constants $\int_{\partial K_i} u \cdot n \, ds$ in Table \ref{sh-check}. From the table, we see that the maximum of the values $\int_{\partial K_i} u \cdot n \, ds$ is almost zero for all cases. This shows that we have exact mass conservation in the coarse grid level.
We remark that we also have fine grid mass conservation by the construction of the basis functions. 

Figure \ref{u-sh-1} and Figure \ref{u-sh-2} shows the corresponding solution plots for Example 1 and Example 2 in perforated domain with small inclusions, where we compare the fine scale velocity solution with different coarse scale velocity solution. In Figure \ref{u-sh-1}, we take 8 and 16 basis functions per coarse element for coarse scale computations. We observe that some fine scale features are lost in the solution when we take 8 basis, and the frame of the coarse edges can be seen in the figure. However, when we take 16 basis, we can observe a much smoother solution which capture the fine features well. Similar behavior can be found in Figure \ref{u-sh-2}, where we observe higher contrast between 4 basis per element and 16 basis per element coarse scale solutions.

\begin{table}[!h]
\centering
  \begin{tabular}{ |c | c | c | c | c | c | }
    \hline
$M_{\text{off}}$ & $DOF$  & $||e_u||_{L^2} $	& $||e_u||_{DG} $ & $||e_u||_{H^1} $ &$||e_p||_{L^2}$   \\  \hline 
\multicolumn{6}{|c|}{Non-oversampling}  \\ \hline
4  & 1280		&33.2  &96.8  &76.8  &--  	 \\  \hline
8  & 2080		&6.5 	 &48.8  	&43.7  &38.1  	 \\  \hline
16  & 3680	    &2.6 	 &31.9  	&28.9  &12 	 \\  \hline
32  & 6880	    &1.9 	 &28.3  	&25.3  &12\\  \hline
\multicolumn{6}{|c|}{Oversampling, $K^+ = K + 4$}  \\ \hline
4  & 1280		&32.6 	 &85.7  	&69.9  &--  	 \\  \hline
8  & 2080		&6.6 	 &39.6  	&36.7  &23.4  	 \\  \hline
16  & 3680	    &1.9 	 &21.7  	&19.4  &2.7 	 \\  \hline
32  & 6880	    &1.8 	 &20.3  	&18.5  &2.7\\  \hline
  \end{tabular}
  \caption{Stokes problem in perforated domain with small inclusions. Numerical results for \textit{Example 1}. Non-oversampling and oversampling with 4 fine layers. }
  \label{err-sh-1}
\end{table}

\begin{table}[!h]
\centering
  \begin{tabular}{ |c | c | c | c | c | c | }
    \hline
$M_{\text{off}}$ & $DOF$  & $||e_u||_{L^2} $	& $||e_u||_{DG} $ & $||e_u||_{H^1} $ &$||e_p||_{L^2}$   \\  \hline
\multicolumn{6}{|c|}{Non-oversampling}  \\ \hline 
4  & 1280		&39.9 	&87.6  	&71.2  &69.8  	 \\  \hline
8  & 2080		&7.6 	 &49.4  	&39.5  &5.0  	 \\  \hline
16  & 3680	    &6.7 	 &36.7  	&31.8  &2.6 	 \\  \hline
32  & 6880	    &4.9 	 &35.9  	&30.5  &2.9\\  \hline
\multicolumn{6}{|c|}{Oversampling, $K^+ = K + 4$}  \\ \hline
4  & 1280		&31.7 	 &69.6  	&52.6  &--  	 \\  \hline
8  & 2080		&2.6 	 &36.7  	&27.8  &16.8  	 \\  \hline
16  & 3680	    &1.8 	 &25.5  	&20.7  &3.6 	 \\  \hline
32  & 6880	    &1.5 	 &20.3  	&17.4  &3.5\\  \hline
  \end{tabular}
  \caption{Stokes problem in perforated domain with small inclusions. Numerical results for \textit{Example 2}. Non-oversampling and oversampling with 4 fine layers. }
  \label{err-sh-2}
\end{table} 

\begin{table}[!h]
\centering
  \begin{tabular}{ |c | c | c | c | }
    \hline
 \multicolumn{4}{|c|}{Example 1}  \\ \hline  
$M_{\text{off}}$ & DOF  & Non-oversampling	& Oversampling  \\  \hline 
4  & 1280		&2.9e-20  &-4.4e-22 	 \\  \hline
8  & 2080		&6.6e-18 &-4.2e-18 	 \\  \hline
16  & 3680	    &5.7e-19 &-9.5e-18 	 \\  \hline
32  & 6880	    &-4.0e-18 &1.2e-15\\  \hline
32  & 6880	    &-4.0e-18 &1.2e-15\\  \hline
\multicolumn{4}{|c|}{Example 2}  \\ \hline
$M_{\text{off}}$ & DOF  & Non-oversampling	& Oversampling  \\  \hline 
4  & 1280		&-8.4e-22  &4.1e-22 	 \\  \hline
8  & 2080		&-1.3e-19  &-1.9e-20 	 \\  \hline
16  & 3680	    &1.9e-19 &-5.8e-22 	 \\  \hline
32  & 6880	    &9.7e-20 &-5.1e-18 \\  \hline
  \end{tabular}
  \caption{Stokes problem in perforated domain with small inclusions. Verification of local mass conservation on coarse edges by computing the maximum of $\int_{\partial K_i} u \cdot n \, ds$ over all coarse blocks. Top: Example 1. Bottom: Example 2}
  \label{sh-check}
\end{table}

\subsection{Perforated domain with some extremely small inclusions}
In this section, we show the numerical results for the Stokes problem in perforated domain with various size of inclusions (right of Figure~\ref{fig:mesh}), see Table \ref{err-ex-1} for Example 1 and Table \ref{err-ex-2} for Example 2. The fine degrees of freedoms for this domain is $91588$, and the coarse degrees of freedoms range only from 680 for 4 basis per element to 3480 for 32 basis per coarse element. Note that, in this domain we use the coarse mesh where each block is a rectangle, thus the coarse degrees of freedom is less than that in the previous section where we used triangular blocks for coarse mesh. From the tables, we can see that for Example 1, the velocity $L^2$ errors can be less than $10\%$ when we take more than 8 basis. Moreover, for Example 2, the velocity $L^2$ errors are already $6.1\%$ (or $3.5\%$) for non-oversampling case (or oversampling case) when we take exactly 8 basis. The convergence results in Table \ref{err-ex-1} indicate that oversampling helps to reduce the energy errors for velocity. For example, we take 32 basis, the velocity $H^1$ error become $12.9\%$ in the oversampling case, which is much smaller than $20.1\%$ in the non-oversampling case. The oversampling strategy works even better to improve the velocity results for Example 2. Table \ref{err-ex-2} shows that the velocity $L^2$, $H^1$ and DG errors are almost reduced by half when we take 8, 16 or 32 basis applying the oversampling strategy. The local mass conservation is also verified by the data presented in Table \ref{ex-check}. Figure \ref{u-ex-1} and Figure \ref{u-ex-2} demonstrate the velocity solution plots for Example 1 and Example 2 respectively. In Figure \ref{u-ex-1}, we compare the fine scale velocity solution with 8 basis coarse scale solution and 16 basis coarse scale solution. It is clear to see that when we take 8 basis, the higher value regions in the solution shrinks, and some properties of the solution between two inclusions are not captured well. These drawbacks are recovered better when we take 16 basis, and the solution is more comparable with fine scale solution. The solution is reported in Figure \ref{u-ex-2} for Example 2, where we compare 4 basis and 16 basis coarse scale solution with fine scale solution. The behavior is similar as before.

In addition, in Figure~\ref{oversamp}, 
we present the comparison the solutions for Example 2 in perforated domain with small inclusions (left of Figure~\ref{fig:mesh}) in oversampling and non-oversampling case respectively. The x-component of velocity is shown on the top, and the y-component is on the bottom, the results for non-oversampling are on the left ($L^2$ error $6.7\%$, $H^1$ error $31.8\%$), and results using oversampling is on the right ($L^2$ error $1.8\%$, $H^1$ error $20.7\%$). Here, we take 16 basis as an example. It can be observed that when we use the oversampling strategy, the transitions from lower values to the higher values in the solution are  smoother compared with the one without oversampling. This helps us to understand the advantage of oversampling visually.

\begin{table}[!h]
\centering
  \begin{tabular}{ |c | c | c | c | c | c | }
    \hline
$M_{\text{off}}$ & $DOF$  & $||e_u||_{L^2} $	& $||e_u||_{DG} $ & $||e_u||_{H^1} $ &$||e_p||_{L^2}$   \\  \hline
\multicolumn{6}{|c|}{Non-oversampling}  \\ \hline 
4  & 680		&46.6 	&93.4  	&79.7  &--  	 \\  \hline
8  & 1080		&11.5 	 &55.0  	&52.1  &39.6  	 \\  \hline
16  & 1880	    &2.9 	 &27.9  	&25.9  &9.1 	 \\  \hline
32  & 3480	    &1.9 	 &22.3  	&20.1  &5.6\\  \hline
\multicolumn{6}{|c|}{Oversampling, $K^+ = K + 4$}  \\ \hline
4  & 680		&50.8 	 &83.3  	&76.3  &--  	 \\  \hline
8  & 1080		&10.8 	 &48.1  	&45.3  &31.6  	 \\  \hline
16  & 1880	    &4.5 	 &23.4  	&21.6  &2.5 	 \\  \hline
32  & 3480	    &1.6 	 &14.5  	&12.9  &2.1\\  \hline
  \end{tabular}
  \caption{Stokes problem in perforated domain with additional small inclusions. Numerical results for \textit{Example 1}. Non-oversampling and oversampling with 4 fine layers. }
  \label{err-ex-1}
\end{table} 

\begin{table}[!h]
\centering
  \begin{tabular}{ |c | c | c | c | c | c | }
    \hline
$M_{\text{off}}$ & $DOF$  & $||e_u||_{L^2} $& $||e_u||_{DG} $ 	& $||e_u||_{H^1} $ &$||e_p||_{L^2}$   \\  \hline 
\multicolumn{6}{|c|}{Non-oversampling}  \\ \hline
4  & 680		&63.1 	 &96.6  	&82.1  &33.6  	 \\  \hline
8  & 1080		&6.1	 &47.7	&36.5  &3.7  	 \\  \hline
16  &1880 	    &3.8 	 &28.4  	&24.2  &1.5 	 \\  \hline
32  & 3480	    &2.9 	 &26.6  	&22.3  &1.4\\  \hline
\multicolumn{6}{|c|}{Oversampling, $K^+ = K + 4$}  \\ \hline
4  & 680		&41.6  &65.6  &54.3  &--  	 \\  \hline
8  & 1080		&3.5 	 &29.3  	&23.1  &11.8  	 \\  \hline
16  &1880	    &1.7 	 &15.5  	&13.0  &4.3 	 \\  \hline
32  &3480 	    &1.3  &12.9  	&11.0  &2.8\\  \hline
  \end{tabular}
  \caption{Stokes problem in perforated domain with additional small inclusions. Numerical results for \textit{Example 2}. Non-oversampling and oversampling with 4 fine layers. }
  \label{err-ex-2}
\end{table}

\begin{table}[!h]
\centering
  \begin{tabular}{ |c | c | c | c | }
    \hline
    \multicolumn{4}{|c|}{Example 1}  \\ \hline
$M_{\text{off}}$ & DOF  & Non-oversampling	& Oversampling  \\  \hline 
4  & 680		&2.3e-20 &2.6e-20 	 \\  \hline
8  & 1080		&1.8e-20 &-5.5e-20 	 \\  \hline
16  & 1880	    &-8.2e-18 & 5.5e-18	 \\  \hline
32  & 2480	    &3.9e-20 &3.5e-17\\  \hline
\multicolumn{4}{|c|}{Example 2}  \\ \hline
$M_{\text{off}}$ & DOF  & Non-oversampling	& Oversampling  \\  \hline 
4  & 680		&1.8e-23 &1.0e-22 	 \\  \hline
8  & 1080		&-5.1e-22 &-1.8e-22 	 \\  \hline
16  & 1880	    &4.7e-19 &1.2e-19 	 \\  \hline
32  & 2480	    &1.4e-20 &-5.2e-21 \\  \hline
  \end{tabular}
  \caption{Stokes problem in perforated domain with small inclusions.  Verification of local mass conservation on coarse edges by computing the maximum value of $\int_{\partial K_i} u\cdot n \, ds$ over all coarse blocks. Top: Example 1. Bottom: Example 2}
  \label{ex-check}
\end{table} 

\begin{figure}[!h]
	\begin{center}
	    \includegraphics[width=0.95\textwidth]{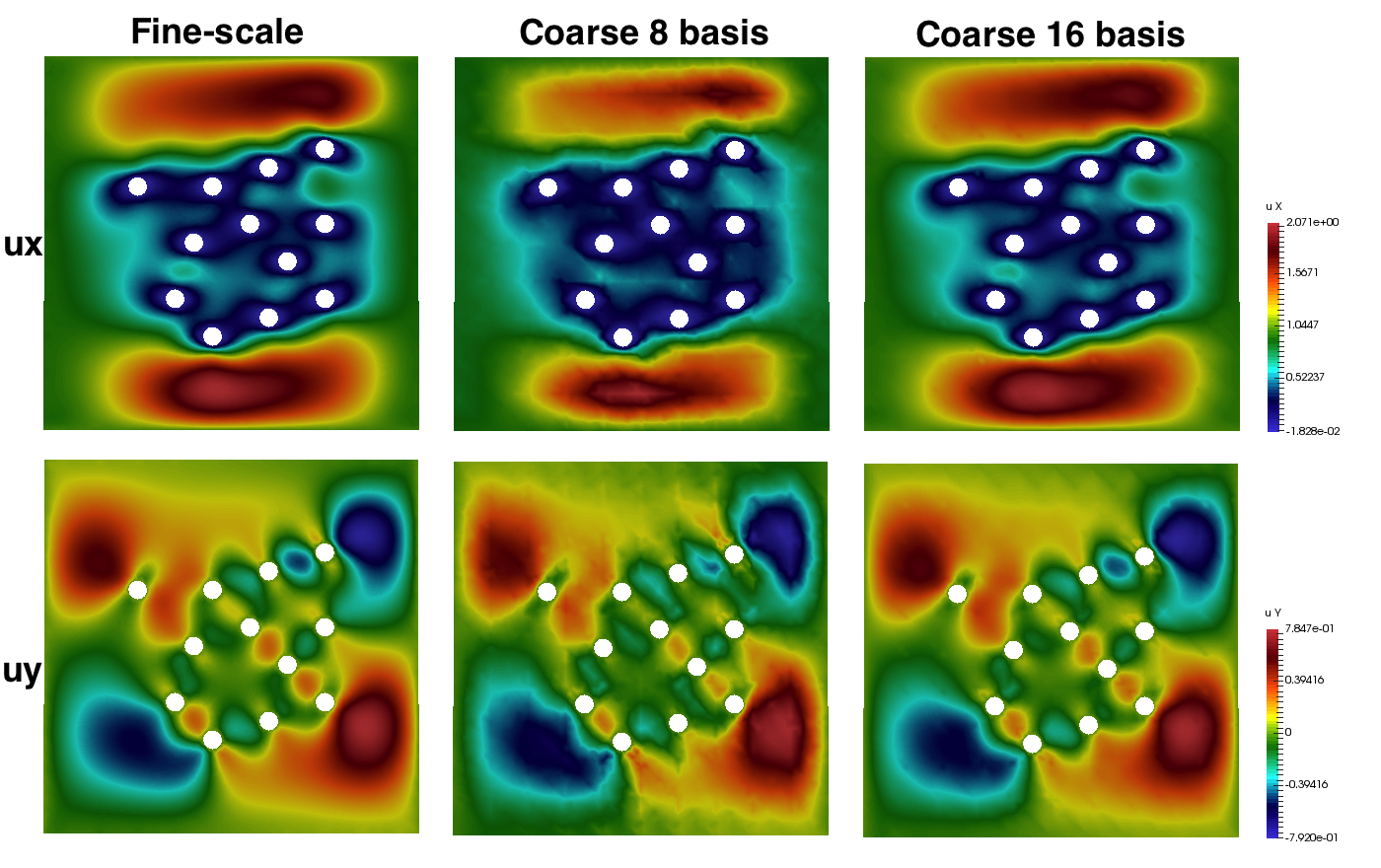}
    \end{center}    
    \caption{Stokes problem for perforated domain with small inclusions. Numerical solution for Example 1. Top: x-component of velocity. Bottom: y-component of velocity.
    Left: Fine-scale solution. 
    Middle: Coarse-scale solution with 8 basis, non-oversampling. 
    Right: Coarse-scale solution with 16 basis, non-oversampling. }
    \label{u-sh-1}
\end{figure}

\begin{figure}[!h]
	\begin{center}
	    \includegraphics[width=0.95\textwidth]{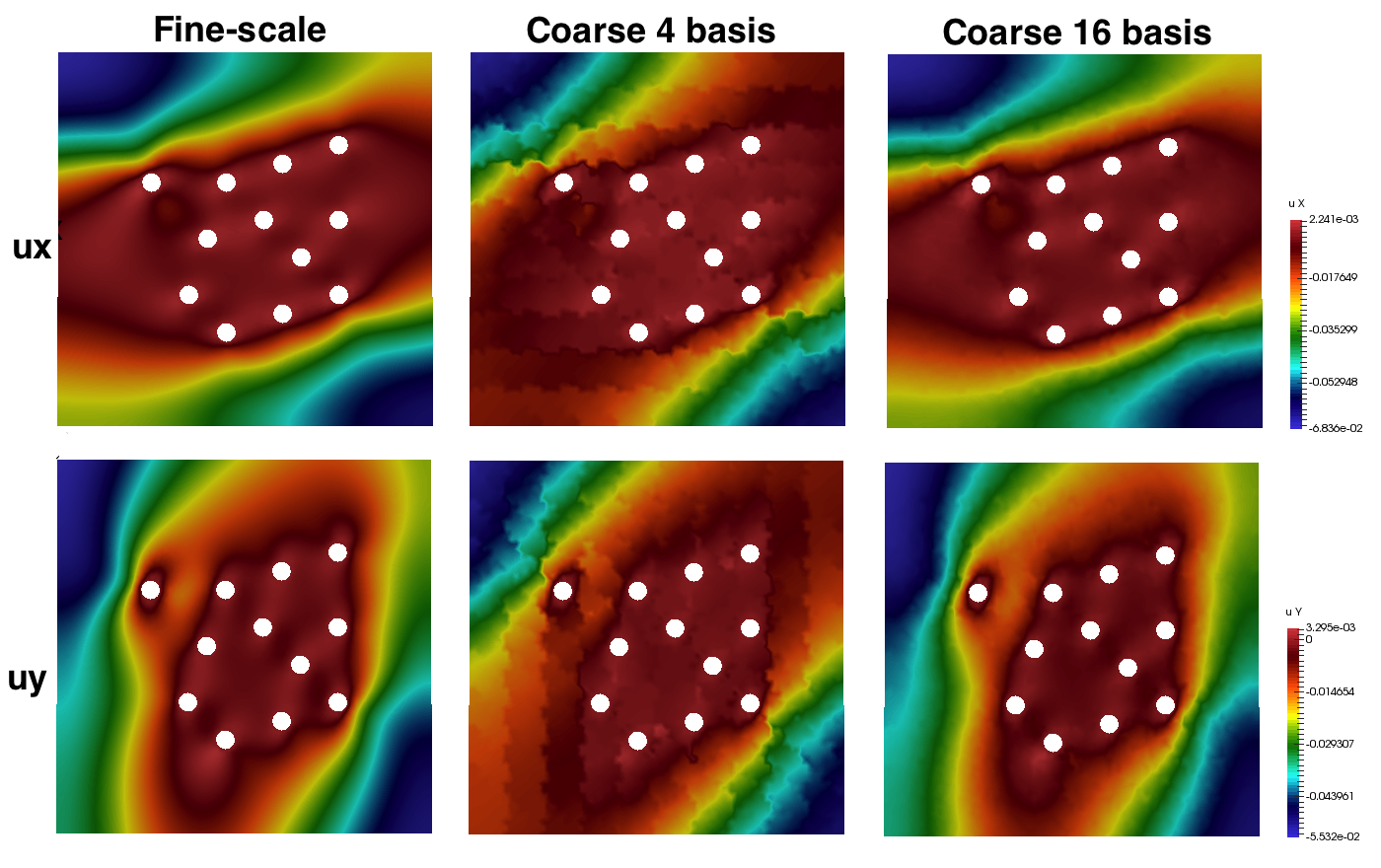}
    \end{center}    
    \caption{Stokes problem for perforated domain with small inclusions. Numerical solution for Example 2. Top: x-component of velocity. Bottom: y-component of velocity.
    Left: Fine-scale solution. 
    Middle: Coarse-scale solution with 4 basis, non-oversampling. 
    Right: Coarse-scale solution with 16 basis, non-oversampling.}
    \label{u-sh-2}
\end{figure}

\begin{figure}[!h]
	\begin{center}
	    \includegraphics[width=0.95\textwidth]{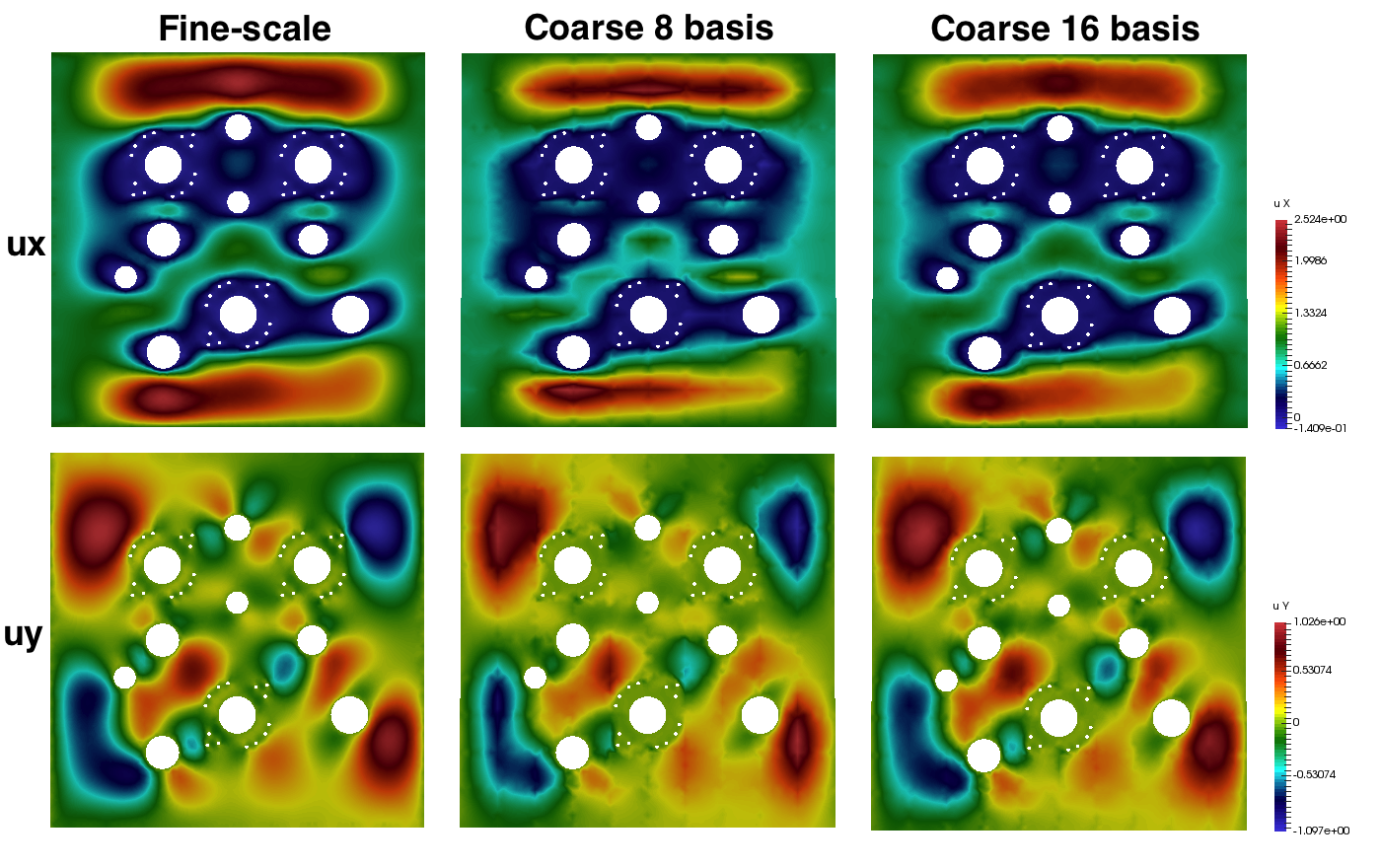}
    \end{center}    
    \caption{Stokes problem for perforated domain with large inclusions. Numerical solution for Example 1. Top: x-component of velocity. Bottom: y-component of velocity.
    Left: Fine-scale solution. 
    Middle: Coarse-scale solution with 8 basis, non-oversampling. 
    Right: Coarse-scale solution with 16 basis, non-oversampling. }
    \label{u-ex-1}
\end{figure}

\begin{figure}[!h]
	\begin{center}
	    \includegraphics[width=0.95\textwidth]{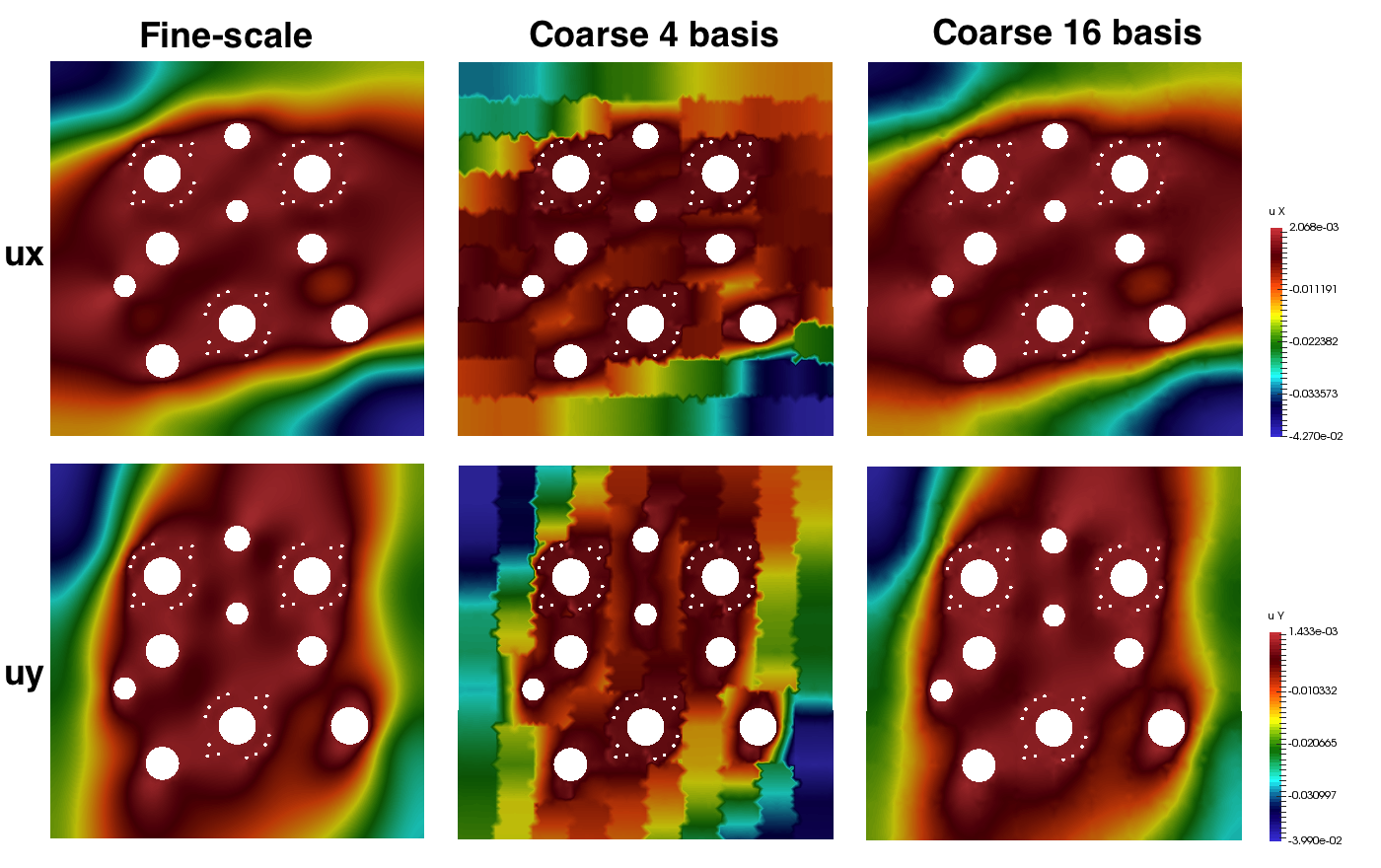}
    \end{center}    
    \caption{Stokes problem for perforated domain with large inclusions. Numerical solution for Example 2. Top: x-component of velocity. Bottom: y-component of velocity.
    Left: Fine-scale solution. 
    Middle: Coarse-scale solution with 4 basis, non-oversampling. 
    Right: Coarse-scale solution with 16 basis, non-oversampling.}
    \label{u-ex-2}
\end{figure}

\begin{figure}[!h]
	\begin{center}
	    \includegraphics[width=0.75\textwidth]{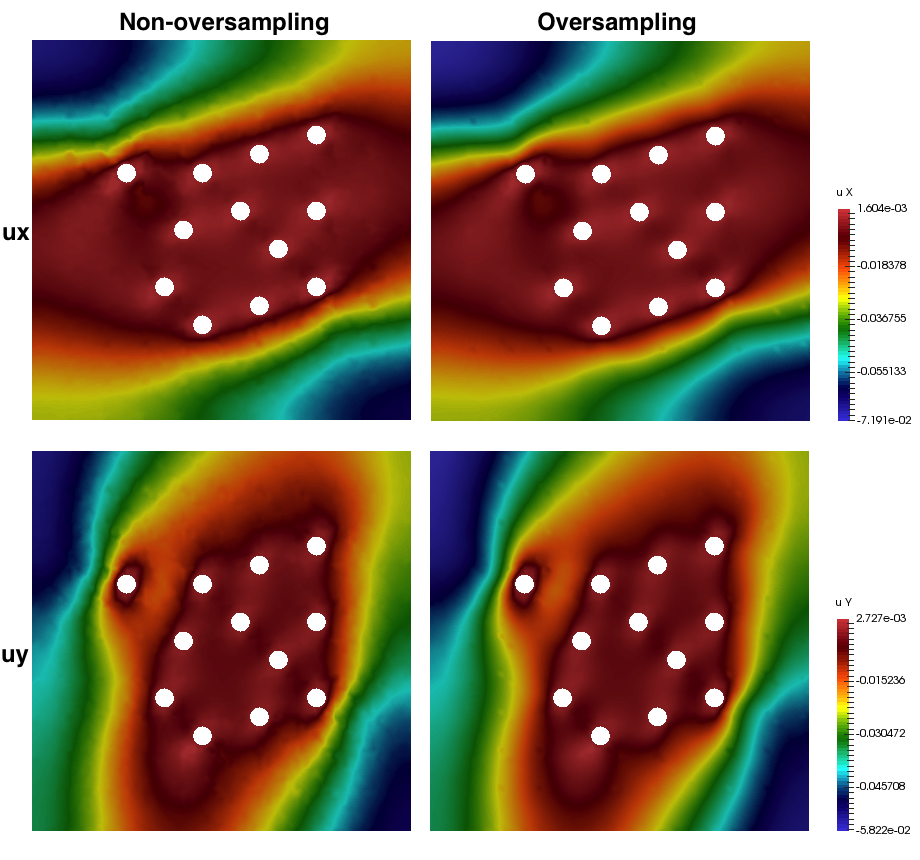}
    \end{center}    
    \caption{Stokes problem for perforated domain with small inclusions. Numerical solution for Example 2. Top: x-component of velocity. Bottom: y-component of velocity.
    Left: Coarse-scale solution with 16 basis, non-oversampling.
    Right: Coarse-scale solution with 16 basis, oversampling with 4 fine layers. }
    \label{oversamp}
\end{figure}

\newpage

\section{Convergence results}\label{sec:analysis}
In this section, we will present the analysis of our multiscale method \eqref{eq:coarse-scale}. First, we will prove the existence and uniqueness of the problem \eqref{eq:coarse-scale} by showing the coercivity and continuity of $a_{\text{DG}}$, the continuity of $b_{\text{DG}}$ and the discrete inf-sup condition for $b_{\text{DG}}$. Next, we will derive a convergence result for our method.
For our analysis, we define the energy norm 
\begin{equation}
\label{eq:A-norm}
\|u\|_A^2 = \int_{\Omega} |\nabla u|^2 + \frac{1}{h} \sum_{E\in\mathcal{E}^H} \int_E |[u]|^2.
\end{equation}
Moreover, we define the following $L^2$ norm
\begin{equation}\label{eq:q-norm}
\|(q, \widehat{q})\|_Q^2= \|q\|_{L^2(\Omega^{\epsilon})}^2 + \sum_{E \in\mathcal{E}^H} h \|\widehat{q}\|_{L^2(E)}^2.
\end{equation}
The notation $\alpha \lesssim \beta$ means that $\alpha \leq C \beta$
for a constant $C$ independent of the mesh size. 
We notice that the $Q$-norm in (\ref{eq:q-norm})
is a weaker norm compared with the more usual choice $\|q\|_{L^2(\Omega^{\epsilon})}^2 + \sum_{E \in\mathcal{E}^H} H \|\widehat{q}\|_{L^2(E)}^2$.

First, we consider the continuity and coercivity of the bilinear form $a_{\text{DG}}$, as well as the continuity of the bilinear form $b_{\text{DG}}$. 
These properties are summarized in the following lemma. 

\begin{lemma}
Assume that $\gamma = O(1)$ is large enough. 
The bilinear form $a_{\text{DG}}$ is continuous and coercive, that is
\begin{eqnarray}
|a_{\text{DG}}(u,v)| &\leq& a_1 \|u\|_A \|v\|_A \\
a_{\text{DG}}(u,u) &\geq& a_0 \|u\|_A^2 \label{eq:coercivity}
\end{eqnarray}
and the bilinear form $b_{\text{DG}}$ is also continuous:
\begin{equation}
|b_{\text{DG}}(v,q,\widehat{q})| \leq b_1 \| v\|_A \|(q,\widehat{q})\|_Q.
\end{equation}
\end{lemma}

\begin{proof}
The proof for continuity and coercivity of $a_{\text{DG}}$ is classical \cite{Laz_PShoebV_2003, AdaptiveGMsDGM, chung2015online}, and will be omitted here.
For the continuity of $b_{\text{DG}}$, it follows from the Cauchy-Schwarz inequality.
\end{proof}

\subsection{Inf-sup condition}
In this section, we will prove an inf-sup condition for the bilinear form $b_{\text{DG}}(v,q,\widehat{q})$. 
We will assume the continuous inf-sup condition holds for $b(v,q)$.  
That is, for any $q\in L^2_0(\Omega^{\epsilon})$,
we have
\begin{equation}
\label{eq:cont-infsup}
\sup_{u\in H^1_0(\Omega^{\epsilon})} \frac{ b(u,q) }{\| u \|_{H^1(\Omega^{\epsilon})}} \geq \beta \|q\|_{L^2(\Omega^{\epsilon})}. 
\end{equation}
We will also assume the following independence condition for the multiscale basis. 
For every coarse block $K_i \in\mathcal{T}^H$, there are at least $4$ basis functions, denoted by $\phi^{i,\text{off}}_j$, $j=1,2,3,4$,  
in the local offline space $V^i_{\text{off}}$
such that there are coefficients $d_{jk}$ such that 
\begin{equation}
\label{eq:assume}
\int_{E_l} \Big( \sum_{j=1}^{4} d_{jk} \phi^{i,\text{off}}_j \Big) \cdot n = \delta_{kl}, \quad k,l=1,2,3,4,
\end{equation}
for all coarse edges $E_l$ on the boundary of $K_i$.
We remark that the above independence condition says that we can construct a function in $V^i_{\text{off}}$
with normal component having mean value one on one coarse edge
and mean value zero on the other coarse edges.
In particular, for each coarse element $K_i$, and for every edge $E_j \in \partial K_i$,
there is a basis function $\Psi_j$ such that $\int_{E_j} \Psi_j \cdot n = 1$
and $\int_{E_k} \Psi_j \cdot n = 0$ for other coarse edges $E_k \in \partial K_i$. 

The next lemma is the main result of this section.
\begin{lemma}
For all $q \in Q_H$ and $\widehat{q} \in \widehat{Q}_H$, we have
\begin{equation}
\label{eq:infsup}
\|(q, \widehat{q} )\|_Q \leq C_{\text{\rm infsup}} \sup_{v \in V_H} \frac{b_{\text{DG}}(v, q, \widehat{q})}{\|v\|_A}
\end{equation}
where $C_{\text{\rm infsup}} > 0$ is a constant independent of the mesh size, provided the fine mesh size $h$ is small enough.
\end{lemma}

\begin{proof}
Let $q \in Q_H$ and $\widehat{q} \in \widehat{Q}_H$ be arbitrary. 
By the continuous inf-sup condition (\ref{eq:cont-infsup}), there is $u\in H^1_0(\Omega^{\epsilon})^2$ such that 
$\div u = q$ and $\| u \|_{H^1(\Omega^{\epsilon})} \leq \beta^{-1} \|q\|_{L^2(\Omega^{\epsilon})}$.
By the assumption (\ref{eq:assume}), for each coarse element $K_i$, and for every edge $E_j \in \partial K_i$,
there is a basis function $\Psi_j$ such that $\int_{E_j} \Psi_j \cdot n = 1$,
and $\int_{E_k} \Psi_j \cdot n = 0$ for other coarse edges $E_k \in \partial K_i$. 
Note that we suppress the dependence of $\Psi_j$ on $i$ to simplify the notations. 
Then we define $v_1 \in V^{\text{off}}$ by
\begin{equation}
\label{eq:v1}
v_1 = \sum_{K_i \in\mathcal{T}^H} \sum_{E_j\in\partial K_i} c_{j,i} \Psi_j, \quad
\text{with} \quad 
c_{j,i} = \int_{E_j} u\cdot n.
\end{equation}
It is clear that 
\begin{equation*}
\int_{E_j} v_1\cdot n = \int_{E_j} u\cdot n, \quad \text{and} \quad
\int_{E_k} v_1 \cdot n = 0.
\end{equation*}
In addition, we define $v_1$ so that $\int_E v_1\cdot n = 0$ for all boundary edges $E\in\partial \Omega^{\epsilon}$.
We also choose the normal vectors in (\ref{eq:v1}) so that the average jumps of $v_1\cdot n$
across all interior coarse edges are zero. This condition can be achieved by choosing a fixed normal direction
for each coarse edge in the definition (\ref{eq:v1}).
By the definition of $b_{\text{DG}}$, integration by parts and using the definition of $v_1$, we have 
\begin{equation*}
b_{\text{DG}}(v_1,q,\widehat{q}) = b(u,q) = \|q\|_{L^2(\Omega^{\epsilon})}^2.
\end{equation*}
Next, we will show that $\|v_1\|_A \leq \alpha \|q\|_{L^2(\Omega^{\epsilon})}$ for some positive constant $\alpha$.
We define the energy $D_j$ of the basis function $\Psi_j$ by
\begin{equation*}
D_j := \int_{K_i} |\nabla \Psi_j|^2 + \frac{1}{h} \int_{\partial K_i} |\Psi_j|^2.
\end{equation*}
So, by the definition of $\|\cdot\|_A$, the trace inequality and the continuous inf-sup condition, 
\begin{equation*}
\| v_1\|_A^2 \leq \sum_{K_i \in\mathcal{T}^H} \sum_{E_j\in\partial K_i} c_{j,i}^2 D_j \lesssim \alpha \|q\|_{L^2(\Omega^{\epsilon})}^2,
\end{equation*}
where we define
\begin{equation*}
\alpha = \max_{K_i\in\mathcal{T}^H} \max_{E_j\in\partial E_i} D_j.
\end{equation*}

On the other hand, we can choose $v_2\in V^{\text{off}}$ such that 
\begin{equation*}
\int_{E_j} v_2\cdot n = \frac{1}{2} (hH) \widehat{q},
\quad\text{and} \quad
\int_{E_k} v_2\cdot n = 0
\end{equation*}
if $E_j$ is an interior edge, or
\begin{equation*}
\int_{E_j} v_2\cdot n =  (hH) \widehat{q},
\quad\text{and} \quad
\int_{E_k} v_2\cdot n = 0
\end{equation*}
if $E_j$ is a boundary edge, 
where $n$ is the outward normal vector on the boundary of $K_i$.
This can be achieved by defining
\begin{equation}
\label{eq:v2}
v_2 = \sum_{K_i \in\mathcal{T}^H} \sum_{E_j\in\partial K_i} d_{j,i} \Psi_j, \quad
\text{with} \quad 
d_{j,i} = \sigma (hH) \widehat{q}
\end{equation}
where $\sigma = 1$ or $\sigma = 1/2$ depending on the location of the coarse edge $E_j$.
Thus, we have 
\begin{equation*}
\int_{E_j} [v_2]\cdot n = (hH) \widehat{q}
\end{equation*}
on all interior coarse edges.
By the definition of $b_{\text{DG}}$, 
\begin{equation*}
b_{\text{DG}}(v_2, q, \widehat{q}) =  -\sum_{K\in\mathcal{T}^H} \int_K q \div v_2 + h\sum_{E\in \mathcal{E}^H} \| \widehat{q}\|_{L^2(E)}^2.
\end{equation*}
We can show that $\|v_2\|^2_A \leq C_1 \alpha (hH) h \sum_{E\in \mathcal{E}^H} \| \widehat{q}\|_{L^2(E)}^2$ using arguments similar as above,
where the constant $C_1$ is independent of the mesh size. 

Finally, we let $v = \alpha_1 v_1 + v_2 \in V^{\text{off}}$. Then
\begin{equation*}
b_{\text{DG}}(v,q,\widehat{q}) = \alpha_1 \|q\|_{L^2(\Omega^{\epsilon})}^2  -\sum_{K\in\mathcal{T}^H} \int_K q \div v_2  + h\sum_{E\in \mathcal{E}^H} \| \widehat{q}\|_{L^2(E)}^2.
\end{equation*}
Using the Young's inequality, we have
\begin{equation*}
b_{\text{DG}}(v,q,\widehat{q}) \geq \alpha_1 \|q\|_{L^2(\Omega^{\epsilon})}^2 - \frac{1}{2C_1\alpha(hH)} \sum_{K\in\mathcal{T}^H} \int_K  \div v_2^2   - \frac{C_1\alpha(hH) }{2} \sum_{K\in\mathcal{T}^H} \int_K q ^2 + h \sum_{E\in \mathcal{E}^H} \| \widehat{q}\|_{L^2(E)}^2
\end{equation*}
which implies
\begin{equation*}
b_{\text{DG}}(v,q,\widehat{q}) \geq (\alpha_1 - \frac{C_1 \alpha (hH)}{2}) \|q\|_{L^2(\Omega^{\epsilon})}^2 +\frac{1}{2}h\sum_{E\in \mathcal{E}^H} \| \widehat{q}\|_{L^2(E)}^2.
\end{equation*}
Taking $\alpha_1 = C_1 \alpha (hH)$ and assuming that the fine mesh size $h$ is small enough so that $C_1 \alpha (hH) = O(1)$, we obtain
\begin{equation*}
b_{\text{DG}}(v,q,\widehat{q}) \geq C \| (q,\widehat{q})\|_Q^2
\end{equation*}
where $C$ is a constant independent of the mesh size. Moreover, 
\begin{equation*}
\|v\|_A^2 \lesssim \alpha_1^2 \|v_1\|_A^2 + \|v_2\|_A^2 \lesssim \alpha_1^2 \alpha \|q\|_{L^2(\Omega^{\epsilon})}^2 + \alpha_1 
h \sum_{E\in \mathcal{E}^H} \| \widehat{q}\|_{L^2(E)}^2.
\end{equation*}
Thus, choosing $h$ small enough, we have $\|v\|_A^2 \lesssim \| (q,\widehat{q})\|_Q^2$.
\end{proof}

\subsection{Convergence results}
In this section, we will derive an error estimate between the fine scale solution $u_h$ and coarse scale solution $u_H$. First, we construct a projection of the fine grid velocity in the snapshot space, and estimate the error for this projection. Second, we will estimate the difference between this projection and coarse scale velocity. Combine these two errors, we obtain the results as desired. 
\begin{theorem}
Let $u_h$ be the fine scale velocity solution in \eqref{eq:fine-scale}, and $u_H$ be the coarse scale velocity solution of \eqref{eq:coarse-scale}. The following estimate holds
\begin{equation*}
\|u_h-u_H \|^2_A  \lesssim
\sum_{i=1}^{N}\frac{H}{\lambda_{L_i+1}^{(i)}}(1+ \frac{H}{h\lambda_{L_i+1}^{(i)}})  \int_{\partial K_i} | (\nabla u_{\text{snap}}) \, n |^2
+ H^2 \|f\|_{L^2(\Omega^{\epsilon})}^2,
\end{equation*}
where $u_{\text{\rm snap}}$ is the snapshot solution defined in \eqref{eq:proj}.
\end{theorem}

\begin{proof}
Let $(u_h, p_h) \in V_h^{\text{DG}} \times Q_H$ be the fine scale solution satisfying \eqref{eq:fine-scale}. 
We will next define a projection, denoted $u_{\text{snap}}$, of $u_h$ in the snapshot space $V_{\text{snap}}$.
For each coarse element $K$, the restriction of $u_{\text{snap}}$ on $K$ is defined by solving 
\begin{equation}\label{eq:proj}
\begin{aligned}
-\Delta u_{\text{snap}} + \nabla p_{\text{snap}} &= 0, \quad &\text{in } K \\
\div u_{\text{snap}} &=c, \quad &\text{in }  K\\
u_{\text{snap}} &= u_h, \quad &\text{on }  \partial{K}
\end{aligned}
\end{equation}
where $p_{\text{snap}}$ is a constant, and
$c$ is chosen by the compatibility condition, $c = \frac{1}{|K|} \int_{\partial K}  u_h \cdot n \, ds$.
We remark that $u_{\text{snap}}$ is obtained on the fine grid, and 
we therefore have $u_{\text{snap}} \in V_{\text{snap}}$. 
We define $u_{\text{off}}$ as the projection of $u_{\text{snap}}$ in the offline space $V_H$.
Using \cite{AdaptiveGMsDGM}, we obtain
\begin{equation}\label{part2}
\|u_{\text{snap}} - u_{\text{off}}\|_A^2
\leq \sum_{i=1}^{N}\frac{H}{\lambda_{L_i+1}^{(i)}}(1+ \frac{H}{h\lambda_{L_i+1}^{(i)}})  \int_{\partial K_i} | (\nabla u_{\text{snap}}) \, n |^2.
\end{equation}

Next, by comparing (\ref{eq:coarse-scale}) and (\ref{eq:fine-scale}), we have
\begin{equation}\label{eq:snap-error}
\begin{aligned}
a_{\text{DG}}(u_{h}-u_H,v) + b_{\text{DG}}(v,p_{h}-p_H,\widehat{p}_{h}-\widehat{p}_H) &= 0,\\
b_{\text{DG}}(u_{h}-u_H,q, \widehat{q}) &= 0,
\end{aligned}
\end{equation}
for all $v \in V_H,  \, q\in Q_H,  \, \widehat{q}\in\widehat{Q}_H$.
Then, using the inf-sup condition (\ref{eq:infsup}) and standard arguments, we have
\begin{equation}
\| u_{h} - u_H \|_A \lesssim \| u_{h} - u_{\text{off}}\|_A.
\end{equation}

Finally, we define $u_h = u_{\text{snap}} + u_0$, where $u_0 = u_h - u_{\text{snap}}$.
Then (\ref{eq:snap-error}) and (\ref{part2}) imply that
\begin{equation}
\| u_{h} - u_H \|_A^2 \lesssim
\sum_{i=1}^{N}\frac{H}{\lambda_{L_i+1}^{(i)}}(1+ \frac{H}{h\lambda_{L_i+1}^{(i)}})  \int_{\partial K_i} | (\nabla u_{\text{snap}}) \, n |^2 + \|u_0\|^2_A.
\end{equation}
By (\ref{eq:fine-scale}), we have
\begin{equation}
\label{eq:u0-1}
a_{\text{DG}}(u_0,v) = -a_{\text{DG}}(u_{\text{snap}},v) + 
(f,v)+  \int_{\Gamma_D} \Big(\frac{\gamma}{h} g_D \cdot v -  ((\nabla v) \, n) \cdot g_D \Big) - b_{\text{DG}}(v,p_h,\widehat{p}_h)
\end{equation}
for all $v\in V_h^{\text{DG}}$.
By the definition of $u_0$, we see that $u_0 = 0$ on $\partial K$ for all coarse element $K\in\mathcal{T}^H$.
Thus, using (\ref{eq:coercivity}) and taking $v=u_0$ in (\ref{eq:u0-1}), we have
\begin{equation}
\| \nabla u_0\|_{A}^2 \lesssim -a_{\text{DG}}(u_{\text{snap}},u_0) + 
(f,u_0).
\end{equation}
Notice that
\begin{equation}
(f,u_0) = \sum_{K\in\mathcal{T}^H} \int_K f\, u_0
\leq \sum_{K\in\mathcal{T}^H} \|f\|_{L^2(K)} \, \|u_0\|_{L^2(K)}
\lesssim H \sum_{K\in\mathcal{T}^H} \|f\|_{L^2(K)} \, \|\nabla u_0\|_{L^2(K)}
\end{equation}
where the last inequality follows from the Poincare inequality. So, we obtain
\begin{equation}
\label{eq:L1}
(f,u_0) \lesssim H \|f\|_{L^2(\Omega^{\epsilon})} \, \|u_0\|_A.
\end{equation}
By the definition of $a_{\text{DG}}$ and $u_0$, we have
\begin{equation*}
a_{\text{DG}}(u_{\text{snap}},u_0) = \int_{\Omega^{\epsilon}} \nabla u_{\text{snap}} : \nabla u_0
- \sum_{E\in\mathcal{E}^H} \int_E \{ (\nabla u_0)  n\}  \cdot [u_{\text{snap}}].
\end{equation*}
Notice that $[u_{\text{snap}}] = [u_h]$ for all $E$. Thus, by the results in \cite{AdaptiveGMsDGM}, we obtain
\begin{equation}
\label{eq:L2}
\sum_{E\in\mathcal{E}^H} \int_E \{ (\nabla u_0)  n\}  \cdot [u_{\text{snap}}]
\lesssim \|u_0\|_A \, \Big( \frac{1}{h} \sum_{E\in\mathcal{E}^H} \int_E |[u_h]|^2\Big)^{\frac{1}{2}}. 
\end{equation}
By the variational form of (\ref{eq:proj}), we have, for all coarse elements $K$
\begin{equation*}
\int_{K} \nabla u_{\text{snap}} : \nabla u_0
= \int_K p_{\text{snap}} \, \div u_0 = 0
\end{equation*}
since $p_{\text{snap}}$ is a constant and $u_0=0$ on $\partial K$.
Combining the above results, we have
\begin{equation}
\| u_0\|_A^2 \lesssim \sum_{i=1}^{N}\frac{H}{\lambda_{L_i+1}^{(i)}}(1+ \frac{H}{h\lambda_{L_i+1}^{(i)}})  \int_{\partial K_i} | (\nabla u_{\text{snap}}) \, n |^2
+ H^2 \|f\|_{L^2(\Omega^{\epsilon})}^2.
\end{equation}
This completes the proof.

\end{proof}

\section{Conclusion}

In this paper, we develop a new GMsFEM for Stokes problems in perforated domains.
The method is based on a discontinuous Galerkin formulation, and constructs
local basis functions for each coarse region. 
The construction of basis follows the general framework of GMsFEM 
by using local snapshots and local spectral problems.
In addition, we use a hybridized technique in order to achieve mass conservation.
Our numerical results show that only a few basis functions per coarse region are needed
in order to obtain a good accuracy.
We also show numerically that the multiscale solution satisfies the mass conservation property.
Furthermore, we prove the stability and the convergence of the scheme.
In the future, we plan to develop adaptivity ideas \cite{chung2016online,chung2016goal} for this method.

\section{Acknowledgement}
EC's research is partially supported by Hong Kong RGC General Research Fund (Project: 400813)
and CUHK Faculty of Science Research Incentive Fund 2015-16.
MV's  work is partially supported by Russian Science Foundation Grant RS 15-11-10024 and  RFBR 15-31-20856.

\bibliographystyle{siam}
\bibliography{references}

\end{document}